\documentclass[runningheads]{llncs}
\usepackage{graphicx}
\usepackage{amsmath,amssymb}
\usepackage{lscape}
\begin{document}
\title{Hyperbolic Optimization over the Integer Efficient Set of MOILFP}
%
%
\author{Fatma Zohra OUAIL\inst{1}\
\and
Mohamed El-Amine CHERGUI\inst{2}
}
\authorrunning{FZ. OUAIL et al.}
%
\institute{USTHB, Faculty of Mathematics, Dpt. of Operational Research
BP 32, Bab Ezzouar 16111, Algiers, Algeria\\\email{ouail.fatmazohra@yahoo.fr}  \and  USTHB, RECITS Laboratory, Faculty of Mathematics
BP 32, Bab Ezzouar 16111, Algiers, Algeria \\
\email{mchergui@usthb.dz}}
\maketitle              
\begin{abstract}
The aim of this study is to find the optimum of a linear fractional function over the efficient set of a multi-objective linear fractional   integer program without generating all efficient solutions. By its nature, it is a global optimization problem since the efficient set is discrete, hence not convex. For this purpose, a branch and bound based method is described with a double mission to search for an optimal solution for a given linear fractional function which is moreover, efficient for a multi-objective linear fractional integer programming problem. Tests performed on instances randomly generated up to 120 variables, 100 constraints and 6 criteria are successful.

\keywords{Fractional programming \and Integer programming \and Multi-objective linear fractional optimization \and Global optimization \and Branch and bound.}

\end{abstract}

\section{Introduction}
Multi-objective programming has gained considerable momentum over the last twenty years given the development of notions and dedicated tools to enable the implementation of methods for solving these problems. Chronologically, researchers have focused first on the exact and approximate methods for multi-objective linear problems, mainly bi-objective problems in the criterion space, both in the continuous case (MOLP)  \cite{29,21,2,30,31,33}, the integer case (MOILP) \cite{28,34,36,27,26,25,22} and the mixed case  (MOMLP) \cite{3,35}. Attempts have been made to generalize exact methods of combinatorial optimization (OC) problems to the multi-objective combinatorial optimization problems (MOCO) \cite{37,38}, but generally unsuccessful. It was not until long after that we focused on the development of exact methods for multi-objective linear fractional problems (MOLFP) \cite{39} and with integer variables (MOILFP) \cite{23,9}. Based on the standard Markovitch model for the portfolio problem, Steuer et al. have studied the multi-objective case optimizing simultaneously the variance (quadratic function) and the mean (linear function) \cite{40,41,42}. More recently, an exact method has been developed for quadratic multi-objective problems with discrete variables (MOIQP) \cite{43}. However, for the multi-objective problems, the number of efficient solutions, or non dominated vectors, can be infinite or finite but astronomical for the MOCO problems and the judicious choice of an appropriate solution becomes very difficult and even impossible for the decision-maker in this case. Therefore studies have emerged in this new context, which consists in fixing a goal of obtaining optimum satisfaction among all the efficient solutions of a multi-objective given problem without enumerating the efficient set. This non-convex optimization problem has been evoked for the first time by Philip in 1972 \cite{8} and since then, the researchers have focused on continuous problems \cite{1,2,3,4,5,6,7} that benefit from an interesting property, namely that the efficient set is connected and contains extreme points of the polyhedron. However, very few works have emerged (been published) for problems considering integer variables. 

In fact, a bibliographic search for this topic would reveal that only four methods have been proposed since 1992, in the linear case (optimizing a linear function over a MOILP), see Nguyen \cite{17}, Abbas and al. \cite{16}, Jorge \cite{11} and finally Bolland and al.\cite{15}. In addition, the domain of optimizing a fractional criterion over the efficient set of a MOILP is still unexplored. We cite for instance \cite{12} where the authors developed simple pivoting techniques to generate progressively integer solutions tested at each step for efficiency and adding some constraints that eliminate not interesting points from the admissible region. And in \cite{121} where the authors developed a branch and bound based technique with the principle of efficient cut to find the optimal solution.
  
Along the present article, a more general field of research than those evoked so far, called the optimization of a hyperbolic criterion vector over the efficient set of a MOILFP will be investigated. This problem, to our knowledge has not been treated before, given its difficulty (NP-hard). In addition, this field took its importance from practice, where many real world situations can be modeled as a MOILFP program and for which we seek another fractional criterion to be optimized different from those considered in the initial problem. For instance, the indicators which are used at evaluation of economic activities are generally in fractional form, such as; output/worker, stock/selling, benefit/cost, etc. Also, Anna Isabel Barros reported in \cite{49} some interesting applications like fractional cutting stock problem \cite{44}, minimal ratio spanning tree \cite{47}, fractional knapsack problems [45], and fractional allocation problems [46] appear to be interesting from a practical point of view. Also, Erik B. Bajalinov cited in his book \cite{48} some applications of integer fractional linear programming problems such that: knapsack  problem, capital budgeting problems, set covering problem and traveling salesman problem which can be easily modeled as  multiple objective problems considering several conflicting objective functions.

This paper is organized as follows: we begin by giving some definitions and notations in section 2. Followed by, some theoretical results and proofs in section 3. Section 4 is dedicated to the presentation of solution method for optimizing a hyperbolic function over a MOILFP. In section 5, a numerical example is selected to understand the different steps of the proposed methodology.  Also, our algorithm was tested, throughout section 6, on a set of random generated instances which pointed out that till 120 variable and 100 constraints, the problem is solved in a reasonable time. We finish by giving some concluding remarks and prospects in section 7.

\section{Definition and notations}
First of all, we recall some basic proprieties of the linear fractional programming problems in which the objective function is the ratio of two linear functions and the constraints are linear. Such problems can be stated as follows: 

\[(FP)\left\{ \begin{array}{c}
{max\ z\left(x\right)=\frac{p^tx+\alpha }{q^tx+\beta }\ } \\ 
x\in X \end{array}\right.\] 
where \textit{$X$} is a non-empty compact polyhedron defined by   \textit{$X=\left\{x\in {\mathbb{R}}^n|Ax=b,\ \ \ x\geq 0\right\}$}; \textit{$A$} is an \textit{$m\times n$} matrix and \textit{b} is a vector of \textit{${\mathbb{R}}^m$}; \textit{$p$} and \textit{$q$} are vectors of \textit{${\mathbb{R}}^n$}; \textit{$\alpha$} and \textit{$\beta$} are scalars and  \textit{$q^tx+\alpha>0,\ \forall i=1,\dots ,k,\ \ \forall x\ \in X$}.

\subsection{Properties}
\begin{itemize}
\item If an optimal solution for a linear fractional program exists, then an extreme point optimum exists. Furthermore, every local minimum is a global minimum. 
\item A procedure that moves from one extreme point to an adjacent one is a viable approach for solving such a problem
\end{itemize}

\noindent \textbf{Convex-Simplex Method\textit{}}

\noindent Because of the special structure of the objective function $z$, the convex-simplex method simplifies into a minor modification of the simplex method of linear programming. We present in this section a method credited to Cambini et al. \cite{50} for solving $(FP)$:\textit{}

\noindent \textbf{Initialization Step: }Find a starting basic feasible solution $x_0$ to the system $Ax=b,\ x\ge 0$. From the corresponding tableau represented by $x_B+{{(A}^B)}^{-1}A^Nx_N={{(A}^B)}^{-1}b$. Let $r=1$ and go to the Main Step.\textit{}

\noindent \textbf{Main Step\textit{}}

\begin{enumerate}
\item \textbf{\textit{ }}Compute the vector  \textit{${\gamma }^t_N={{\nabla }_Nz(x_r)}^t-{{\nabla }_Bz(x_r)}^t{{(A}^B)}^{-1}A^N$}.

\item  If \textit{${\gamma }_N\le 0$}, stop; the current point \textit{$x_r$} is an optimal solution ,

\item  Otherwise, go to Step 2.

\item  Let  \textit{${\gamma }_j=max\left\{{\gamma }_i,\ i\in N\right\}$}.
\[\frac{{\hat{b}}_s}{{\hat{a}}_{sj}}={\mathop{min}_{1\le i\le m} \left\{\frac{{\hat{b}}_i}{{\hat{a}}_{ij}};\ y_{ij}>0\right\}\ }\] 
\end{enumerate}
where : \textit{$\hat{b}={{(A}^B)}^{-1}b$}, \textit{${\hat{a}}_j={{(A}^B)}^{-1}A^j$};  Go to Step 3

\noindent 

\begin{enumerate}
\item  Replace the variable \textit{$x_{B_s}$} by the variable \textit{$x_j$}, update the tableau correspondingly by pivoting at \textit{${\hat{a}}_{sj}$}. Let the current  solution be \textit{$x_r$}, \textit{$r=r+1$} and go to Step 1.
\end{enumerate}

\noindent We consider now a problem in which $k$ objective functions to be maximized. Each objective function is the ratio of two linear functions and subject to linear constraints. Such problems are called multiple objective linear fractional programming problems and can be stated as follows: \textit{}
\[(P)\left\{ \begin{array}{c}
Max\ z_i\left(x\right)=\frac{p^t_ix+{\alpha }_i}{q^t_ix+{\beta }_i};\ \ \forall i=\overline{1,k} \\ 
x\in D=X\cap{\mathbb{Z}}^n \end{array}
\right.\] 
where $X$ is a non-empty compact polyhedron defined by   $X=\left\{x\in {\mathbb{R}}^n|Ax=b,\ \ \ x=0\right\}$ ;\textit{}, \textit{ $A$} is an \textit{$m\times n$} matrix and \textit{b} is a vector of \textit{${\mathbb{R}}^m$}, \textit{$p_i$} and \textit{$q_i,\ \ \forall i=\overline{1,k}$}, are vectors of \textit{${\mathbb{R}}^n$}, \textit{$\alpha^i$} and \textit{$\beta^i$} are scalars, \textit{$q^t_ix+{\beta }_i>0,\ \forall i=1,\dots ,k,\ \ \forall x\ \in X$}.
\begin{definition}
\noindent A point $x\in D$ is called an efficient solution, or Pareto-optimal solution for ($P)$, if and only if there does not exist another point $y\in D$ such that $z_i\left(y\right)=z_i(x)$  for all $i\in \left\{1,\dots ,k\right\}$  and $z_i\left(y\right)>z_i(x)$ for at least one index  $i\in \left\{1,\dots ,k\right\}$. Otherwise $x$ is not efficient and the vector $z(y)$ dominates the vector $z(x)$, where $z=(z_1,,\dots ,z_k)$.
\end{definition}
\noindent The set of all integer efficient solutions of $(P)$ is denoted by $XE$.

\noindent The problem that we propose to solve consists in optimizing a new linear fractional function over the efficient set $XE$ of $(P)$.  This last set being non convex, the considered problem belongs to the class of global optimization problems. The corresponding mathematical program is written as follows:\textit{}
\[(PE)\left\{ \begin{array}{c}
Max\ \psi (x)=\frac{c^tx+\lambda }{d^tx+\mu } \\ 
x\in XE \end{array}
\right.\] 
where  $c,\ d\ \in {\mathbb{R}}^n$ and $\lambda ,\mu \ \in \mathbb{R}$ with a positive denominator  $d^tx+\mu ,\ \ \forall x\in X$.\textit{}

\noindent Let us denote by $({PE}^l$)  the master program on which relies the method at each stage l:  \textit{}
\[({PE}^l)\left\{ \begin{array}{c}
Max\ \psi (x)=\frac{c^tx+\lambda }{d^tx+\mu } \\ 
x\in X_l \end{array}
\right.\] 

\begin{enumerate}
\item  \textit{$X_0=X\ and\ X_l$} defined below;

\item  \textit{$x^*_l$ }is the integer solution found during solving\textit{ $({PE}^l$) }using eventually branch and bound ;

\item  \textit{$x_{opt}$ }is the best efficient solution of\textit{ $({PE}^l)$ }found till step\textit{ $l$ }and\textit{ ${\psi }_{opt}$ }its corresponding criterion value ;

\item  \textit{By $B_l$ (${respectively\ N}_l$ ), }we mean the set index of basic (respectively non basic) variables of\textit{ $x^*_l\ $};\textit{}

\item $\bar{\gamma}_i$ the reduced gradient vector of the $i$th objective. It
is defined by: $\bar{\gamma}_{i}=\bar{\beta}_{i}\bar{p}_{i}-\bar{\alpha}_{i}\bar{q}_{i}$, where $\bar{\beta}_{i}$, $\bar{p}_{i}$, $\bar{\alpha}_{i}$ and $\bar{q}_{i}$  are updated values obtained from the optimal simplex tableau of $(PE^l)$;

\item \textit{ ${\mathit{\Delta}}_l=\left\{j\in N_l|\ \exists \ i=\overline{1,k},\ {\bar{\gamma} }_i>0\right\}\cup \left\{j\in N_l|{\bar{\gamma} }_i=0\ for\ all\ i=\overline{1,k}\right\}$} describes the set of all the improving directions for the criteria of the program \textit{$(P)$};

\item \textit{ }Cut of type\textit{ I }
\[\sum_{j\in {\mathit{\Delta}}_l}{x_j\ge 1}\ \] 

to remove non efficient solutions without having to enumerate them and the corresponding set \textit{$X^1_{l+1}=\left\{x\in X_l|\sum_{j\in {\mathit{\Delta}}_l}{x_j\ge 1}\right\}\ $};

\item  Cut of type\textit{ II }is constructed according to the following inequality : $$\psi \left(x\right)\ge \ {\psi }_{opt}\ $$

to allow removing uninteresting points regarding optimality and the corresponding set \textit{$X^2_{l+1}=\left\{x\in X_l|\psi \left(x\right)\ge {\psi }_{opt}\right\}\ $}  ;
\item  The resulting set from \textit{$X^1_{l+1}$} and \textit{$X^2_{l+1}$} is given by : \textit{$X_{l+1}=X^1_{l+1}\cup X^2_{l+1}$};\textit{}

\item \textit{ }The ideal point \textit{${id}^l=({id}_1,\dots ,{id}_k)$ }for the program\textit{$\ ({PE}^l$) }whose coordinate are given by\textit{ ${id}_i={max \left\{z_i(x)/x\in X_l\right\},\ \ i=\overline{1,k}\ }$}.\textit{}
\end{enumerate}
\noindent \textbf{Efficiency test\textit{}}

\noindent There are two options to test the efficiency of an integer solution found at stage $l,\ x^*_l$:

\begin{enumerate}
\item  Update\textit{ }the\textit{ }set of potentially efficient solution \textit{${XE}_l$} by introducing\textit{ $x^*_l$.} The solution \textit{$x^*_l$ }is not efficient if its criterion vector is dominated by at least one of those in\textit{ ${XE}_l$}.\textit{}

\item \textit{ }Characterization of Pareto optimal solutions: ~Given  a  point  \textit{$x^*_l\in D$} and let

\noindent 
\[(Q)\ \left\{ \begin{array}{c}
{max e^ts\ } \\ 
st:z\left(x\right)-s=z\left(x^*_l\right) \\ 
x\in D,\  \\ 
s\ge 0 \end{array}
,\right\}\] 

where $e\ $is a vector column of ones.\textit{}

\noindent \textit{}

\begin{theorem}\cite{51}
\noindent $x^*_l$ is efficient if and only if $(Q)$ has a maximum value of zero. Otherwise (the maximum value of  $(Q)$ is finite nonzero), the obtained solution is efficient.
\end{theorem}
\end{enumerate}
\section{Principal of the method}
The proposed algorithm generates the optimal solution of \textit{(PE)} without having to enumerate $XE$ (recall that $XE$ is the efficient set of \textit{(P)}).  Based on branch and bound technique, the method is reinforced by efficient cuts and additional saturating tests allowing a smart search for the optimal solution.  We start by solving the program ($PE^l$ ) using the simplex method at
step $l$ of the algorithm (eventually dual simplex method). Then, to catch how the criteria vectors move
from basis to basis, $k$ lines are added to the basic simplex tableau and reduced costs are calculated in respect of the corresponding basis.  If the obtained solution is non integer, then we still imposing
integrity restrictions on the original variables of the program till getting integer ones. Once an integer
solution $x_l^*$ is achieved, new cuts are established and added to the current simplex tableau which allow reduce the search area considerably (containing non efficient and non interesting solutions for (PE)).
We consider henceforth two types of nodes, those relative to branching process (type 1) and others
to efficient cuts (type 2). So, a node of type 2 is pruned if no improvement of the criteria can be done
along the remaining domain or if an efficient solution is reached at a stage $l$.  A node of type 1 is
fathomed if $\psi_opt$ the best value of $\psi$ obtained till stage $l$ is greater than or equal to value of $\psi$ at that
node, even the corresponding solution is non integer or the domain becomes unfeasible.

\subsection{Algorithm}

\noindent \textbf{Step 0 : } Solve according to $\psi(x)$ with $x\in X$ till optimality (using method in \cite{52}) and then test the resulting solution for efficiency.  If the optimal solution is efficient then Stop, else go to step 1.
\noindent \textbf{Step 1 (initialization): } index $l\ =\ 0$ and the optimal value of the objective function ${\psi }_{opt}\mathrm{=}\mathrm{-}\mathrm{\infty }$\textit{ } to which it corresponds no optimal solution yet ($x_{opt}$  unknown at the beginning). ${XE}_l$denotes the set of potentially efficient solutions of program $(P)$ at each stage $l$ with ${XE}_0\boldsymbol{=}\boldsymbol{\emptyset }$.\textbf{ \textit{}}

\noindent \textbf{Step 2 (main step):} While there is no saturated node in the tree search, solve (${PE}^l$ ) using dual simplex method, except for the initial program$\ ({PE}^0)$ where the Cambini and Martein's method is used.  Go to Step3.1.\textit{}

\noindent \textbf{Step 3 (tests):\textit{}}

\noindent \textbf{Step 3.1 Feasibility test:}  If program (${PE}^l$) is infeasible, then stop and the node l is saturated. Else, let $x^*_l$ be the obtained solution, if ${\psi }_{opt}$ $\ge \psi (x^*_l)$, the node l is fathomed, else, go to step3.2 ;\textbf{\textit{}}

\noindent \textbf{Step 3.2  Integrity test:}  If${\ x}^*_l$ is integer, update ${XE}_l$ and go to step3.3. Else go to step 4 ;\textbf{\textit{}}

\noindent \textbf{Step 3.3  Efficiency test:}  If ${\ x}^*_l\ $is not kept within ${XE}_l$, ${\ x}^*_l\ \ $is not efficient, go to step 5. Else, solve program$(Q)$ for efficiency. If ${\ x}^*_l$ is efficient then update eventually ${\psi }_{opt}$and $x_{opt}$ : ${\psi }_{opt}=\psi (x^*_l)$ and $x_{opt}=x^*_l$, the node l is pruned since no improvement of $\psi $ further. Else, let $y_l\ $be the optimal solution of program$(Q)$, update if necessary ${\psi }_{opt}$, $x_{opt}\ $ as before and the set ${XE}_l$ as well. Go to step 5.\textit{}

\noindent \textbf{Step 4 (branching process):} Choose one coordinate $x_j$  of $x^*_l$ such that $x_j$ = ${\alpha }_j$ , with ${\alpha }_j$ a fractional number. Then, split the program (${PE}^l$) into two sub programs, by adding the constraints $x_j\le \left\lfloor {\alpha }_j\right\rfloor $ to obtain $\left({PE}^{l_1}\right)$ and $x_j\ge \left(\left\lfloor {\alpha }_j\right\rfloor +1\right)$ to obtain $\left({PE}^{l_2}\right)$ such that $l_1>l+1$, $l_2>l+1$, and $l_1\neq l_2$. This has the effect to create two new branches in the search tree. Go to step 2.\textit{}

\noindent In fact, since the search tree is treated according to the principle ``depth first'', the cut of type II $\psi (x)\ \ge \ {\psi }_{opt}$ is added in the second branch $l_2$.\textit{}

\noindent \textbf{Step 5 (efficient cut):}  Construct the set ${\mathit{\Delta}}_l$ and compute the ideal point ${id}_l$.  \textit{}

\noindent If ${\mathit{\Delta}}_l=\emptyset $ or ${id}_l$ is dominated by one vector in ${XE}_l$, the node $l$ is fathomed since no efficient solution exists afterward. Otherwise, construct the set $X^1_{l+1}\ $by adding the cut of type I. Go to step 2.

\section{Illustrative example}
\noindent In this section, we illustrate the different steps of the proposed algorithm by solving the following example:\textit{}
\[\left(PE\right)\left\{ \begin{array}{c}
Max\ \psi \left(x\right)=\frac{66x1+50x2+96x3+56x4+16x5+\ 32x6+\ 2}{38x1+\ 40x2+\ 12x3+\ 47x4\ +\ \ 52\ x5+78x6+13} \\ 
 \\ 
x\in XE \end{array}
\right.\] 
\textit{}

\noindent \textit{}

\noindent Where, $XE$ is the integer efficient set of the following problem : \textit{}

\noindent \textit{}

\noindent (P)$\left\{s.t. \begin{array}{c}
Max\ z1\left(x\right)=\ \frac{99x1+\ \ \ \ 89x2+\ \ \ \ 43x3+\ \ \ \ 99\ x4+\ \ \ 33x5+\ \ \ \ 77x6+9}{21\ x1+\ \ \ \ 7x2+\ \ \ \ 69x3+\ \ \ \ 71x4+\ \ \ \ 70x5+\ \ \ \ 70x6+17} \\ 
Max\ z2\left(x\right)=\frac{67x1+\ \ \ \ 36x2+\ \ \ \ 73x3+\ \ \ \ 48\ x4+\ \ \ 60x5+\ \ \ \ 39x6-8}{39\ x1+\ \ \ 43x2+\ \ \ \ 51\ \ x3+\ \ 30\ \ x4+\ \ 91x5+\ \ \ \ 14x6+18} \\ 
 \\ 
18x1+\ 5x2+\ \ \ 14x3+\ \ \ \ 6x4+\ \ \ \ 12x5+\ \ \ 5x6\le 86 \\ 
1x1+\ \ \ 22x2+\ \ \ \ 21\ x3+\ \ 3\ x4+\ \ \ 28x5+\ \ \ \ 24x6\le 66 \\ 
9x1+\ \ \ \ \ 2x2+\ \ \ \ 17x3+\ \ \ 6\ x4+\ \ 21\ x5+\ \ \ \ 13x6\le 78 \\ 
x\ge 0,\ integer \end{array}
\right.$\textit{}

\noindent \textit{}

\noindent We first solve the problem (PE) till optimality using the method of Seshan and Tikekar\cite{52} Cambini{\dots}, we obtain thus the optimal solution $x^*=(0,0,3,0,0,0)$ which is non efficient. We obtain from the efficiency test the efficient solution $y*=(0,\ 1,\ 0,12,\ 0,\ 0)$. \textit{}

\noindent Then set $x\_opt=(0,\ 1,\ 0,12,\ 0,\ 0)$, $\psi \_opt=1.1734$, ${XE}_0=\left\{(0,\ 1,\ 0,12,\ 0,\ 0)\right\}$ and \textit{}
\[{SND}_0=\left\{(1.4680\ \ \ \ 1.4347)\right\}.\] 
\textit{}

\noindent \textbf{Node 0: }$List=\{0\}$.\textbf{ }\textit{}

\noindent Solve the problem (${PE}^0)\ \left\{ \begin{array}{c}
Max\ \psi \left(x\right)=\frac{66x1+50x2+96x3+56x4+16x5+\ 32x6+\ 2}{38x1+\ 40x2+\ 12x3+\ 47x4\ +\ \ 52\ x5+78x6+13} \\ 
 \\ 
x\in X \end{array}
\right.$\textit{}

\noindent \textit{}

\noindent The optimal solution obtained is $x^*_0=(0,\ 0,\ \frac{22}{7},\ 0,\ 0,\ 0)$ which is not integer. Calculate $\psi \left(x^*_0\right)=5.9887$ which is greater then $\psi \_opt$, so we continue to explore the remaining domain. We should then split the original problem into two sub-problems: (${PE}^1)$ and (${PE}^2)$. In the first we add the constraint: $x_3\le 3$ and in the second the constraint:$x_3\ge 4$.\textit{}
\[List=\{1,2\}\] 
\textbf{Node 1: }solve (\textit{${PE}^1)$}, the obtained solution is: \textit{$x^*_1=(0,\ 0,\ 3,\ 0,\ 0,\ 0)$} which is integer. 

\noindent \textit{B}${}_{1}$= (8,3,9,7), \textit{N}${}_{1}$=(1,2,4,5,6,10). Calculate \textit{$\psi \left(x^*_1\right)=5.9184$} which is greater then \textit{$\psi \_opt$}, so we continue to explore the remaining domain. We update \textit{${XE}_1$} and \textit{${SND}_1$}, the solution \textit{$x^*_1\ $} is dominated by (0,1,0,12,0,0). So construct the set \textit{${\Delta }_1=\{1,2,4,6\}$} and compute 

\noindent \textit{${id}^1=(7.2632,\ \ 1.7566)$} which is not dominated by any vector in \textit{${SND}_1$} . Add the cut \textit{$x1+x2+x4+x6\ge 1$} and use dual simplex method to solve the new problem. The optimal solution found is: \textit{$x^*_2=(1,0,3,0,0,0)$} which is integer. \textit{$\psi \left(x^*_2\right)=4.0920$} which is greater than \textit{$\psi \_opt$}. We continue to explore the remaining domain. By updating \textit{${XE}_2$} and \textit{${SND}_2$}, we notice that \textit{$x^*_2$} is not efficient. Construct \textit{${\Delta }_2=\{2,6,10,11\}$} and compute \textit{${id}^2=(7.2632,\ \ 1.7566)$} which is not dominated by any vector in \textit{${SND}_2$}. Add the cut : \textit{$x2+x6+x10+x11\ge 1$} and by using the dual simplex method we obtain the following solution:\textit{$x^*_3=(19/42,\ \ 23/42,\ \ 107/42,\ 0,\ 0,\ 0)$}. \textit{$\psi \left(x^*_3\right)=3.6751$} which is greater than \textit{$\psi \_opt$} , we continue to explore the remaining domain. We split the actual problem into two sub-problems corresponding to \textbf{node 4} and \textbf{node 5. }\textit{$List=\{4,5,2\}$}.\textbf{}

\noindent \textbf{Node 4: }Add the constraint \textit{$x_1\le 0$}, the optimal solution  thus obtained using the dual simplex method is: \textit{$x^*_4=(0,\ 21/40,\ 101/40,\ 19/40,\ 0,\ 0)$} which is not integer, \textit{$\psi \left(x^*_4\right)=3.4315$} which is greater than \textit{${\psi }_{opt}$}so we continue to explore the remaining domain. We split this problem into two new ones: (\textit{${PE}^6)$} and \textit{$({PE}^7)$}. \textit{$List=\{6,7,5,2\}$}.\textbf{}

\noindent \textbf{Node 6: }add the constraint \textit{$x_2\le 0$}, \textit{$x^*_6=\left(0,\ 0,\ 2,\ 1,\ 0,\ 0\right)$} which is integer and not efficient. Update \textit{${XE}_6$} and \textit{${SND}_6$} . \textit{$\psi \left(x^*_6\right)=2.9762$} which is greater than \textit{${\psi }_{opt}$} so we continue to explore the remaining domain. \textit{N}${}_{6}$=$\{$5, 6, 11, 12, 13, 14$\}$, \textit{B}${}_{6}$=$\{$2, 3, 9, 7, 1, 10, 4, 8$\}$. Construct set \textit{${\Delta }_6=\left\{6,\ 11,\ 12,\ 14\right\}$} and compute \textit{${id}^6=\left(1.3787,\ \ 1.7566\right)$} which is not dominated by any vector in \textit{${SND}_6$}. Add the efficient cut : \textit{$x6+x11+x12+x14\ge 1$}.The resulting optimal solution is :\textbf{ }\textit{$x^*_7=\left(0,\ 0,\ \frac{43}{15},\ \frac{\ 29}{15},\ 0,\ 0\right)$} which is not integer. \textit{$\psi \left(x^*_7\right)=2.7878$} which is greater than \textit{${\psi }_{opt}$}, so we continue to explore the remaining domain. We split this problem into two new ones : (\textit{${PE}^8$}) and (\textit{${PE}^9$}). \textit{$List=\left\{8,\ 9,\ 7,\ 5,\ 2\right\}$}.

\noindent \textbf{Node 8:} add the constraint \textit{$x_3\le 2$}, the optimal solution found\\ \textit{$x^*_8=(0,\ 0,\ 2,\ 3/2,\ 0,\ 0)$} which is not integer.  \textit{$\psi \left(x^*_8\right)=\ 2.5860$} which is greater than \textit{${\psi }_{opt}$} so we continue to explore the remaining domain. We split this problem into two new ones: (\textit{${PE}^{10})$} and \textit{$({PE}^{11})$}. \textit{$List=\{10,11,9,7,5,2\}$}.

\noindent \textbf{Node 10: }add the constraint \textit{$x_4\le 1$}, the optimal solution thus found is \textit{$x^*_{10}=(0,\ 0,\ 2,\ 1,\ 0,\ \ 1/4)$}. \textit{$\psi \left(x^*_{10}\right)=2.4928$} which is greater than \textit{${\psi }_{opt}$} so we continue to explore the remaining domain. We split this problem into two new ones: (\textit{${PE}^{12})$} and \textit{$({PE}^{13})$}. \textit{$List=\{12,13,11,9,7,5,2\}$}.\textbf{}

\noindent \textbf{Node 12: }add the constraint \textit{$x_6\le 0$} , the optimal solution found is \textit{$x^*_{12}\boldsymbol{=}(0,\ 0,\ 1,\ 1,\ 0,\ 0)$} which is integer. \textit{B}${}_{12}$=$\{$12,3,1,7,2,10,,8,11,9,6,16$\}$,\\ \textit{N}${}_{12}$=$\{$5,13,1,15,17,18$\}$. \textit{$\psi \left(x^*_{12}\right)=\ 2.1389$} which is greater than \textit{${\psi }_{opt}$} so we continue to explore the remaining domain. Update \textit{${XE}_{12}$}\textbf{ }and \textit{${SND}_{12}$}\textbf{. }\textit{$x^*_{12}$} is not efficient . Construct \textit{${\Delta }_{12}=\{15,17,18\}$} and compute \textit{}

\noindent id${}^{12}$=\textit{$(1.2273,\ 1.1414)$} which is dominated. We prune this node.\textbf{}

\noindent \textbf{Node 13: }add the constraint \textit{$x6\ge 1$}, the optimal solution thus found is \textit{$x^*_{13}=$}(0, 0,13/7,           1, 0, 1) which is not optimal. \textit{$\psi \left(x^*_{13}\right)=\ 1.6738$} which is greater than \textit{${\psi }_{opt}$} so we continue to explore the remaining domain. Split this node into two new nodes 14 and 15. 
\[List\ =\{14,15,11,9,7,5,2\}.\] 
\textbf{Node 14: }add the constraint\textbf{ }$x3\le 1$, the optimal solution thus found is $x^*_{14}=(0,\ 0,\ 1,\ 1,\ \ 0,\ 7/4)$ which is not integer. Compute $\psi \left(x^*_{14}\right)=\ 1.0072$ which is less than ${\psi }_{opt}$ so we prune this node.

\noindent \textbf{Node 15: }add the constraint $x3\ge 2$. The optimal solution is $x^*_{15}=(0,\ 0,\ 2,\ 0,\ 0,\ 1)$ which is integer and not efficient.  Compute $\psi \left(x^*_{15}\right)=\ 1.9652$ which is greater than ${\psi }_{opt}$ so we continue to explore the remaining domain. N${}_{15}$=$\{$2,5,8,13,18,19$\}$, B${}_{15}$=$\{$12,3,1,7,16,10,4,14,11,9,6,15,17$\}$, ${\Delta }_{15}=\emptyset $. ${XE}_{15}$ and ${SND}_{15}$ remains as before with no change.  As ${\Delta }_{15}$ is empty, we break this node. \textit{}

\noindent \textbf{Node 11: }add the constraint $x_4\ge 2$ , we continue till unfeasibility.

\noindent \textbf{Node 9: }add the constraint \textit{$x_3\ge 3$} , we continue till unfeasibility.

\noindent \textbf{Node 7: }add the constraint \textit{$x_2\ge 1$}, the optimal solution  is \textit{$x^*_7=(0,0,2,2,0,0)$} which is integer. \textit{N}${}_{7}$=(1,2,5,6,16,17), \textit{B}${}_{7}$=(12,3,13,7,14,10,4,8,11,9,15). Compute \textit{$\psi \left(x^*_7\right)=2.3359$} which is greater than \textit{${\psi }_{opt}$} so we continue to explore the remaining domain. \textit{${\Delta }_7=\{1,2,6,16,17\}$} add the efficient cut \textit{$x_1+x_2+x_6+x_{16}+x_{17}\ge 1$}. The optimal solution thus found is \textit{$x^*_{16}=(0,\ 0,\ 2,\ 3,\ 0,\ 0)$}. \textit{$\psi \left(x^*_{16}\right)=2.0337$} which is greater than \textit{${\psi }_{opt}$} so we continue to explore the remaining domain. No update for \textit{${XE}_7$} and \textit{${SND}_{16}$}, the integer solution \textit{$x^*_{16}$} is not efficient. Compute \textit{${\Delta }_{16}$}=$\{$6,1,16,18$\}$, \textit{N}${}_{16}$=$\{$5,6,13,14,16,18$\}$, \textit{B}${}_{16}$=$\{$12,3,17,7,2,10,4,8,11,9,15,1$\}$. \textit{${Id}_{16}$} is not dominated by any vector in \textit{${SND}_{16}$}.  Add  the cut \textit{$x_6+x_1+x_{16}+x_{18}\ge 1$}, the optimal solution is \textit{$x^*_{17}=(0,0,2,4,0,0)$} which is integer.\textit{$\ \psi \left(x^*_{17}\right)=1.8578\ $}which is greater than \textit{${\psi }_{opt}$} so we continue to explore the remaining domain. \textit{${XE}_{17}$} and \textit{${SND}_{17}$} with no change. Construct \textit{${\Delta }_{17}=\{2,6,19\}$}. Compute \textit{${Id}_{17}$}, it is not dominated by any vector in \textit{${SND}_{17}$}. Add the cut \textit{$x_2+x_6+x_{19}\ge 1$}. The optimal solution is \textit{$x^*_{18}=(0,\ 0,\ 2,\ 5,\ 0,\ 0)$} which is integer. \textit{$\psi \left(x^*_{18}\right)=1.7426\ $} which is greater than \textit{${\psi }_{opt}$} so we continue to explore the remaining domain. \textit{${\Delta }_{18}=\{6,14,16,20\}$}. Compute \textit{${Id}_{18}$}, which is not dominated. we continue till founding \textit{$x^*_{20}=(0,\ 0,\ 2,\ 7,\ 0,\ 0)$} which is not efficient (effcient test) but we found an efficient  one from resolving \textit{$(PE\_x^*_{20})\ )$}  which is \textit{$y^*=(4,1,0,0,0,1)$} . Update \textit{${\psi }_{opt}$} and  \textit{$x_{opt}$}.

\noindent \textit{${\boldsymbol{x}}_{\boldsymbol{opt}}\boldsymbol{=(}\boldsymbol{4},\boldsymbol{1},\boldsymbol{0},\boldsymbol{0},\boldsymbol{0},\boldsymbol{1}\boldsymbol{)}$}\textbf{ }and\textbf{ }\textit{${\boldsymbol{\psi }}_{\boldsymbol{opt}}\boldsymbol{=}\boldsymbol{1},\boldsymbol{2297}.$}  \textit{${\Delta }_{20}=\{6,14,16,22\}$} and compute \textit{${Id}_{20}$} which is not dominated. We continue in this manner until we get \textit{${\boldsymbol{x}}_{\boldsymbol{opt}}=(4,0,0,2,0,0)$}, \textit{${\boldsymbol{\psi }}_{\boldsymbol{opt}}=1.4595$}.

\noindent \textbf{Node5: }add the cut $x_1\ge 1$ and continue till infeasibility.\textbf{\textit{}}

\noindent \textbf{Node 2: }add the cut $x_3\ge 4$ and continue till infeasibility.\textbf{\textit{}}

\noindent \textbf{Summary: }${\boldsymbol{x}}_{\boldsymbol{opt}}\boldsymbol{=}\boldsymbol{(}\boldsymbol{4}\boldsymbol{,\ }\boldsymbol{0}\boldsymbol{,\ }\boldsymbol{0}\boldsymbol{,\ }\boldsymbol{0}\boldsymbol{,\ }\boldsymbol{0}\boldsymbol{,\ }\boldsymbol{0}\boldsymbol{)}$\textbf{ }and\textbf{ }${\boldsymbol{\psi }}_{\boldsymbol{opt}}\boldsymbol{=}\boldsymbol{1}.\boldsymbol{6121}$\textbf{\textit{}}

\section{Theoretical results}
In order to justify the different steps of the proposed algorithm, the following results are established.  We denote by $D_l$ the set $D_l=X_l\cap {\mathbb{Z}}^n$.\textit{}

\begin{theorem} Suppose that ${\mathit{\Delta}}_l\neq \emptyset $ at the current integer solution $x^*_l$. If $x\neq x^*_l$ is an optimal solution of program (${PE}^l$) in domain $X_l$, then $x\in X_{l+1}$.
\end{theorem}
\begin{proof}
\textit{Let }$x\ x_{l}^{}$\textit{ be an integer solution in 
domain }$X_{l}$\textit{ such that }$x \notin X_{l+1}$ then $
x\notin X_{l+1}^{1}\vee x\notin X_{l+1}^{2}$..

\begin{itemize}
\item \textit{If }$x \notin X_{l+1}^{1}$\textit{, then}$x \in \lbrace 
x \in X_{l}|\sum_{j \in N_{l}\backslash\Delta _{l}}{x_{j}\ge 1}\rbrace $
\textit{. Therefore, the coordinates of }$x$\textit{ check the 
following inequalities: }$\sum_{j \in \Delta _{l}}{x_{j}<1}$\textit{ and 
}$\sum_{j \in N_{l}\backslash\Delta _{l}}{x_{j}\ge 1}$\textit{. It 
follows that }$x_{j}=0$\textit{ for all }$j \in \Delta _{l}$\textit{
. Using the simplex table in }$x_{l}^{}$\textit{ , the following 
equality is supported by all criterion }$i\in \lbrace 1,\ldots ,k\rbrace $
\item \textit{From the simplex table corresponding to the optimal 
solution }$x_{l}^{}$\textit{, the criteria are evaluated by:}
\end{itemize}
\begin{center}$
z_{i}(x)=\frac{\sum_{j \in N_{l}}{\bar{c}_{j}^{i}x_{j}+\bar{\alpha 
}^{i}}}{\sum_{j\in N_{l}}{\bar{d}_{j}^{i}x_{j}+\bar{\beta }^{i}}} $
\textit{ for }$i \in \lbrace 1,\ldots ,k\rbrace $\end{center}

\textit{Where}

$\frac{\bar{\alpha }^{i}}{\bar{\beta }^{i}}=z_{i}(x_{l}^{})$\\

\textit{Then we can write }

\begin{center}$z_{i}(x)=\frac{\sum_{j\in N_{l}\backslash\Delta 
_{l}}{\bar{c}_{j}^{i}x_{j}+\bar{\alpha 
}^{i}}}{\sum_{j\in N_{l}\backslash\Delta 
_{l}}{\bar{d}_{j}^{i}x_{j}+\bar{\beta }^{i}}} $\textit{ for }$
i\in \lbrace 1,\ldots ,k\rbrace $\end{center}

\textit{In the other hand , }$\bar{\gamma }_{j}^{i}=\bar{\beta 
}^{i}\bar{c}_{j}^{i}-\bar{\alpha }^{i}\bar{d}_{j}^{i}\le 0$\textit{, 
for all index }$j \in  N_{l}\backslash\Delta _{l}$\textit{ and }$
\bar{\gamma }_{j}^{i}=\bar{\beta }^{i}\bar{c}_{j}^{i}-\bar{\alpha 
}^{i}\bar{d}_{j}^{i}<0$\textit{ for at least one criterion, implies 
that }$\bar{c}_{j}^{i}-\bar{\alpha }^{i}\bar{d}_{j}^{i}/\bar{\beta 
}^{i}$\textit{ for all }$j \in N_{l}\backslash\Delta _{l}$\textit{ 
because }$\bar{\beta }^{i}=d^{i}x_{l}^{}+\beta ^{i}>0$

\textit{ For all criterion }$i \in \lbrace 1,\ldots ,k\rbrace $\textit{. 
The decision variables being nonnegative, we obtain}

\textit{ }$\bar{c}_{j}^{i}x_{j}\le (\frac{\bar{\alpha 
}^{i}\bar{d}_{j}^{i}}{\bar{\beta }^{i}})x_{j}$\textit{ for all }$
j \in N_{l}\backslash\Delta _{l}$\textit{ and hence:}

$\sum_{j\in N_{l}\backslash\Delta _{l}}{\bar{c}_{j}^{i}x_{j}} \in \le 
\sum_{j \in N_{l}\backslash\Delta _{l}}{(\frac{\bar{\alpha 
}^{i}\bar{d}_{j}^{i}}{\bar{\beta 
}^{i}})x_{j}}\sum_{j \in N_{l}\backslash\Delta 
_{l}}{\bar{c}_{j}^{i}x_{j}}+\bar{\alpha }^{i}\in \le 
\sum_{j \in N_{l}\backslash\Delta _{l}}{(\frac{\bar{\alpha 
}^{i}\bar{d}_{j}^{i}}{\bar{\beta }^{i}})x_{j}}+\bar{\alpha }^{i}$\\

\textit{For any criterion }$z_{i}, i\in \lbrace 1,\ldots ,k\rbrace $
\textit{, the following inequality is obtained:}

$z_{i}(x)=\frac{\sum_{j\in N_{l}\backslash\Delta 
_{l}}{\bar{c}_{j}^{i}x_{j}}+\bar{\alpha 
}^{i}}{\sum_{j \in N_{l}\backslash\Delta _{l}}{(\frac{\bar{\alpha 
}^{i}\bar{d}_{j}^{i}}{\bar{\beta }^{i}})x_{j}}+\bar{\beta }^{i}}$\\

$\Rightarrow z_{i}(x)\le \frac{\sum_{j\in N_{l}\backslash\Delta 
_{l}}{(\frac{\bar{\alpha }^{i}\bar{d}_{j}^{i}}{\bar{\beta 
}^{i}})x_{j}}+\bar{\alpha }^{i}}{\sum_{j\in N_{l}\backslash\Delta 
_{l}}{\bar{d}_{j}^{i}x_{j}}+\bar{\beta }^{i}}$\\

$\Rightarrow z_{i}(x)\le \frac{\frac{\bar{\alpha }^{i}}{\bar{\beta 
}^{i}}(\sum_{j\in N_{l}\backslash\Delta 
_{l}}{\bar{d}_{j}^{i}x_{j}}+\bar{\alpha 
}^{i})}{\sum_{j_in N_{l}\backslash\Delta 
_{l}}{\bar{d}_{j}^{i}x_{j}}+\bar{\beta }^{i}}$\\

$\Rightarrow z_{i}(x)\le \frac{\bar{\alpha }^{i}}{\bar{\beta }^{i}}$\\

$ \Rightarrow z_{i}(x)\le z_{i}(x_{l}^{})$\\

\textit{Consequently, }$z_{i}(x)\le z_{i}(x_{l}^{})$\textit{ for 
all }$i_in \lbrace 1,\ldots ,k\rbrace $\textit{ and }$
z_{i}(x)<z_{i}(x_{l}^{})$\textit{ for at least one index }$i$
\textit{. Hence, }$z(x_{l}^{})$\textit{ dominates }$z(x)$
\textit{ and the solution }$x$\textit{ is not efficient.}

\begin{itemize}
\item \textit{If }$x\in X_{l+1}^{2}$\textit{, }$x$\textit{ is not 
optimal, contradiction, \ldots (**)}
\item \textit{From (*) and (**) ; we conclude that }$x\in X_{l+1}$
\end{itemize}
\end{proof}
\begin{theorem}
Let $x_l^*$ be the current integer solution of program $(PE^l)$, then if $x_l^*$ is efficient for program (P), then it is an optimal solution of program (PE) over $D_l$.
\end{theorem}
\begin{proof}

Suppose that $x_l^*$ is not optimal for program ($PE$). Then, $\exists x\in D_{l},\ x\neq x_l^*$ such that $\varphi(x)>\varphi_{opt}$ from \textbf{Theorem 1}. However, $x_l^*$ being efficient, which means that $\varphi_{opt}\geq \varphi(x_l^*)$. Thus $\varphi(x)\geq\varphi(x_l^*)$. In the other hand, at the current simplex tableau, the expression of $\varphi$ can be written as: 

$\varphi(x)=\varphi(x_l^*)+\sum\limits_{j\in N_{l}}\widehat{\varphi}_{j}x_{j}$

$\Rightarrow$ $\varphi(x_l^*)+\sum\limits_{j\in N_{l}}\widehat{\varphi}_{j}x_{j}>\varphi(x_l^*)$

$\Rightarrow$ $\sum\limits_{j\in N_{l}}\widehat{\varphi}_{j}x_{j}>0$ 

which contradicts the fact that $\widehat{\varphi}_{j}\leq 0,\ \forall j\in N_{l}$.
\end{proof}

\begin{proposition}
If $\Delta_l=\varnothing$, then $\forall x \in D_{l+1}$, $x$ is not efficient.
\end{proposition}
\begin{proof}
$\Delta_l=\emptyset$, then $\forall i\in\left\{1,\cdots,r\right\}$, $\forall j\in N_l$, we have $\hat{c}^i_j\leq 0$ and $\exists i_{0}\in \left\{1,\cdots,r\right\}$ such that $\hat{c}^{i_{0}}_j < 0$ $\forall j\in N_l$. So, $x_l^*$ dominates all points $x$, $x\neq x_l^*$ of domain $D_l$.
\end{proof}

\begin{proposition}
If $\Delta_l=\emptyset$ at the current integer solution $x_l^*$, then   $D_l\backslash{x_l^*}$ is an unexplored domain.
\end{proposition}
\begin{proof}
$\Delta_l=\emptyset$ means that $x_l^*$ is an optimal solution for all criterion,  hence $x_l^*$ is an ideal point  in the domain $X_l$ and $X_l\backslash{x_l^*}$. does not contain efficient solutions.
\end{proof}
\begin{theorem}
If $\psi_opt\geqslant\psi(x_l^*)$, then $\nexists x\in D_l$   such that $\psi(x)>\psi_opt$.
\end{theorem}
\begin{proof}
It is obvious that all solutions $x$ for which $\psi(x)<\psi_{opt}$ are not interesting even efficient, since the existence of  an efficient solution giving already the best value of $\psi$,
\end{proof}
\begin{theorem}
The algorithm terminates in a finite number of iterations and returns the optimal solution of program (PE).
\end{theorem}
\begin{proof}
contains a finite number of integer solutions. At each step $l$ of the algorithm, if an integer solution $x_{l}^{\ast }$ is reached, we proceed to eliminate it as well as a subset of integer non interesting solutions by taking into account \textbf{Theorem 1} above (adding cuts). In the other hand, four saturating tests are used without loss of the optimal solution of (PE). First, when the set $\Delta_{l}$ is empty the corresponding solution $x_{l}^{\ast }$ is an ideal point and the current node can be pruned since no criterion can be improved. Secondly, if at a stage $l$, the current integer solution $x_l^*$ is efficient, the corresponding node is fathomed since $x_l^*$ is optimal for (PE) over $D_l$. Third, if $\psi_{opt}$ (value of the best efficient solution found for (PE)) is greater than that of the optimal solution over $D_l$, the node $l$ also is fathomed. Finally, the trivial case when the reduced domain becomes infeasible. Hence, the algorithm converges toward an optimal solution for (PE) in finite number of steps.
\end{proof}
\section{Experimental study}

The proposed method has been coded using MATLAB R2013a and run on a personal computer with 3.4 GHz Dell Pentium 4 and 4 GB of memory. We should notice that all subroutines were programmed and no optimization packages were used. 

The data are uncorrelated integers uniformly distributed in the interval [1, 99] for the numerator and denominator coefficients, [-10, 20] for the numerator constant, [1, 20] for the denominator constant, and [1, 30] for constraints coefficients. For each constraint, the right-hand side value was taken into [50,100]. For each instance (n,m,k) (n is the number of variables, m the number of constraints, and k the number of objectives), a series of 10 problems were solved. Tables 1 and 2 summarize the obtained results where mean, min and maximum number of CPU (s), number of efficient solutions (Eff Sol) are reported. Also the ratio of SN (staurated node)/ CN (created node) is given for each triplet.
\newpage
\begin{landscape}

\begin{table}[htbp]
  \centering
  \caption{Medium size problems}
    \begin{tabular}{ccccccccccc}
    \hline
    \multicolumn{1}{c}{\textbf{n}} & \multicolumn{1}{c}{\textbf{m}} & \multicolumn{1}{c}{\textbf{r}} & \multicolumn{1}{c}{\textbf{Mean CPU(s)}} & \multicolumn{1}{c}{\textbf{Min CPU(s)}} & \multicolumn{1}{c}{\textbf{Max CPU(s)}} & \multicolumn{1}{c}{\textbf{Mean Eff Sol}} & \multicolumn{1}{c}{\textbf{Min Eff Sol}} & \multicolumn{1}{c}{\textbf{Max Eff Sol}} & \multicolumn{1}{c}{\textbf{Mean CN}} & \multicolumn{1}{c}{\textbf{Mean SN/CN \%}} \\
    \hline
    \multicolumn{1}{l}{\textbf{30}} & \multicolumn{1}{l}{\textbf{10}} & \multicolumn{1}{l}{\textbf{2}} & \multicolumn{1}{l}{\textbf{30.88}} & \multicolumn{1}{l}{\textbf{0.19}} & \multicolumn{1}{l}{\textbf{127.44}} & \multicolumn{1}{l}{\textbf{3.80}} & \multicolumn{1}{l}{\textbf{1}} & \multicolumn{1}{l}{\textbf{7}} & \multicolumn{1}{l}{\textbf{2165,6}} & \multicolumn{1}{l}{\textbf{41.65}} \\
    \multicolumn{1}{l}{\textbf{30}} & \multicolumn{1}{l}{\textbf{10}} & \multicolumn{1}{l}{\textbf{3}} & \multicolumn{1}{l}{\textbf{18,39}} & \multicolumn{1}{l}{\textbf{0.35}} & \multicolumn{1}{l}{\textbf{79.21}} & \multicolumn{1}{l}{\textbf{8.50}} & \multicolumn{1}{l}{\textbf{3}} & \multicolumn{1}{l}{\textbf{17}} & \multicolumn{1}{c}{\textbf{1921,6}} & \multicolumn{1}{l}{\textbf{31.29}} \\
    \multicolumn{1}{l}{\textbf{30}} & \multicolumn{1}{l}{\textbf{10}} & \multicolumn{1}{l}{\textbf{4}} & \multicolumn{1}{l}{\textbf{28.68}} & \multicolumn{1}{l}{\textbf{0.28}} & \multicolumn{1}{l}{\textbf{125.85}} & \multicolumn{1}{l}{\textbf{12.40}} & \multicolumn{1}{l}{\textbf{3}} & \multicolumn{1}{l}{\textbf{35}} & \multicolumn{1}{l}{\textbf{2282,5}} & \multicolumn{1}{l}{\textbf{32.48}} \\
    \multicolumn{1}{l}{\textbf{30}} & \multicolumn{1}{l}{\textbf{15}} & \multicolumn{1}{l}{\textbf{2}} & \multicolumn{1}{l}{\textbf{22.76}} & \multicolumn{1}{l}{\textbf{0.17}} & \multicolumn{1}{l}{\textbf{97.03}} & \multicolumn{1}{l}{\textbf{4.42}} & \multicolumn{1}{l}{\textbf{1}} & \multicolumn{1}{l}{\textbf{13}} & \multicolumn{1}{l}{\textbf{1971.00}} & \multicolumn{1}{l}{\textbf{38.12}} \\
    \multicolumn{1}{l}{\textbf{30}} & \multicolumn{1}{l}{\textbf{15}} & \multicolumn{1}{l}{\textbf{3}} & \multicolumn{1}{l}{\textbf{13.40}} & \multicolumn{1}{l}{\textbf{0.15}} & \multicolumn{1}{l}{\textbf{51.91}} & \multicolumn{1}{l}{\textbf{7.58}} & \multicolumn{1}{l}{\textbf{1}} & \multicolumn{1}{l}{\textbf{17}} & \multicolumn{1}{l}{\textbf{1193.95}} & \multicolumn{1}{l}{\textbf{34.23}} \\
    \multicolumn{1}{l}{\textbf{30}} & \multicolumn{1}{l}{\textbf{15}} & \multicolumn{1}{l}{\textbf{4}} & \multicolumn{1}{l}{\textbf{25.05}} & \multicolumn{1}{l}{\textbf{0.24}} & \multicolumn{1}{l}{\textbf{132.01}} & \multicolumn{1}{l}{\textbf{11.40}} & \multicolumn{1}{l}{\textbf{3}} & \multicolumn{1}{l}{\textbf{23}} & \multicolumn{1}{l}{\textbf{1951.00}} & \multicolumn{1}{l}{\textbf{31.07}} \\
    \multicolumn{1}{l}{\textbf{30}} & \multicolumn{1}{l}{\textbf{25}} & \multicolumn{1}{l}{\textbf{2}} & \multicolumn{1}{l}{\textbf{31.08}} & \multicolumn{1}{l}{\textbf{0.33}} & \multicolumn{1}{l}{\textbf{152.00}} & \multicolumn{1}{l}{\textbf{4.40}} & \multicolumn{1}{l}{\textbf{2}} & \multicolumn{1}{l}{\textbf{9}} & \multicolumn{1}{l}{\textbf{1833.20}} & \multicolumn{1}{l}{\textbf{38.01}} \\
    \multicolumn{1}{l}{\textbf{30}} & \multicolumn{1}{l}{\textbf{25}} & \multicolumn{1}{l}{\textbf{3}} & \multicolumn{1}{l}{\textbf{17.13}} & \multicolumn{1}{l}{\textbf{0.62}} & \multicolumn{1}{l}{\textbf{83.69}} & \multicolumn{1}{l}{\textbf{6.10}} & \multicolumn{1}{l}{\textbf{2}} & \multicolumn{1}{l}{\textbf{10}} & \multicolumn{1}{l}{\textbf{1276.20}} & \multicolumn{1}{l}{\textbf{36.99}} \\
    \multicolumn{1}{l}{\textbf{30}} & \multicolumn{1}{l}{\textbf{25}} & \multicolumn{1}{l}{\textbf{4}} & \multicolumn{1}{l}{\textbf{13.46}} & \multicolumn{1}{l}{\textbf{0.54}} & \multicolumn{1}{l}{\textbf{50.95}} & \multicolumn{1}{l}{\textbf{9.70}} & \multicolumn{1}{l}{\textbf{3}} & \multicolumn{1}{l}{\textbf{19}} & \multicolumn{1}{l}{\textbf{804.90}} & \multicolumn{1}{l}{\textbf{32.77}} \\
    \multicolumn{1}{l}{\textbf{30}} & \multicolumn{1}{l}{\textbf{55}} & \multicolumn{1}{l}{\textbf{2}} & \multicolumn{1}{l}{\textbf{19.87}} & \multicolumn{1}{l}{\textbf{0.26}} & \multicolumn{1}{l}{\textbf{116.60}} & \multicolumn{1}{l}{\textbf{4.32}} & \multicolumn{1}{l}{\textbf{2}} & \multicolumn{1}{l}{\textbf{11}} & \multicolumn{1}{l}{\textbf{760.11}} & \multicolumn{1}{l}{\textbf{37.49}} \\
    \multicolumn{1}{l}{\textbf{40}} & \multicolumn{1}{l}{\textbf{60}} & \multicolumn{1}{l}{\textbf{3}} & \multicolumn{1}{l}{\textbf{38,28}} & \multicolumn{1}{l}{\textbf{0.82}} & \multicolumn{1}{l}{\textbf{176.71}} & \multicolumn{1}{l}{\textbf{7.32}} & \multicolumn{1}{l}{\textbf{2}} & \multicolumn{1}{l}{\textbf{21}} & \multicolumn{1}{l}{\textbf{806.05}} & \multicolumn{1}{l}{\textbf{34.02}} \\
    \multicolumn{1}{l}{\textbf{50}} & \multicolumn{1}{l}{\textbf{10}} & \multicolumn{1}{l}{\textbf{2}} & \multicolumn{1}{l}{\textbf{48.33}} & \multicolumn{1}{l}{\textbf{1.04}} & \multicolumn{1}{l}{\textbf{127.63}} & \multicolumn{1}{l}{\textbf{4.00}} & \multicolumn{1}{l}{\textbf{1}} & \multicolumn{1}{l}{\textbf{10}} & \multicolumn{1}{l}{\textbf{4053.50}} & \multicolumn{1}{l}{\textbf{36.76}} \\
    \multicolumn{1}{l}{\textbf{50}} & \multicolumn{1}{l}{\textbf{10}} & \multicolumn{1}{l}{\textbf{3}} & \multicolumn{1}{l}{\textbf{120.10}} & \multicolumn{1}{l}{\textbf{0.16}} & \multicolumn{1}{l}{\textbf{354,91}} & \multicolumn{1}{l}{\textbf{11.30}} & \multicolumn{1}{l}{\textbf{1}} & \multicolumn{1}{l}{\textbf{34}} & \multicolumn{1}{l}{\textbf{5753,8}} & \multicolumn{1}{l}{\textbf{31.07}} \\
    \multicolumn{1}{l}{\textbf{50}} & \multicolumn{1}{l}{\textbf{10}} & \multicolumn{1}{l}{\textbf{4}} & \multicolumn{1}{l}{\textbf{127.61}} & \multicolumn{1}{l}{\textbf{0.17}} & \multicolumn{1}{l}{\textbf{539,2}} & \multicolumn{1}{l}{\textbf{9.10}} & \multicolumn{1}{l}{\textbf{1}} & \multicolumn{1}{l}{\textbf{28}} & \multicolumn{1}{l}{\textbf{4569,8}} & \multicolumn{1}{l}{\textbf{34.91}} \\
    \multicolumn{1}{l}{\textbf{50}} & \multicolumn{1}{l}{\textbf{15}} & \multicolumn{1}{l}{\textbf{2}} & \multicolumn{1}{l}{\textbf{324.61}} & \multicolumn{1}{l}{\textbf{43.23}} & \multicolumn{1}{l}{\textbf{1116.76}} & \multicolumn{1}{l}{\textbf{6.60}} & \multicolumn{1}{l}{\textbf{3}} & \multicolumn{1}{l}{\textbf{14}} & \multicolumn{1}{l}{\textbf{12486.00}} & \multicolumn{1}{l}{\textbf{45.55}} \\
    \multicolumn{1}{l}{\textbf{50}} & \multicolumn{1}{l}{\textbf{15}} & \multicolumn{1}{l}{\textbf{3}} & \multicolumn{1}{l}{\textbf{124.24}} & \multicolumn{1}{l}{\textbf{2.30}} & \multicolumn{1}{l}{\textbf{419.01}} & \multicolumn{1}{l}{\textbf{7.90}} & \multicolumn{1}{l}{\textbf{3}} & \multicolumn{1}{l}{\textbf{20}} & \multicolumn{1}{l}{\textbf{5481.00}} & \multicolumn{1}{l}{\textbf{37.46}} \\
    \multicolumn{1}{l}{\textbf{50}} & \multicolumn{1}{l}{\textbf{15}} & \multicolumn{1}{l}{\textbf{4}} & \multicolumn{1}{l}{\textbf{44.93}} & \multicolumn{1}{l}{\textbf{0.78}} & \multicolumn{1}{l}{\textbf{147.84}} & \multicolumn{1}{l}{\textbf{11.50}} & \multicolumn{1}{l}{\textbf{5}} & \multicolumn{1}{l}{\textbf{25}} & \multicolumn{1}{l}{\textbf{2617.80}} & \multicolumn{1}{l}{\textbf{33.12}} \\
    \multicolumn{1}{l}{\textbf{50}} & \multicolumn{1}{l}{\textbf{25}} & \multicolumn{1}{l}{\textbf{2}} & \multicolumn{1}{l}{\textbf{76.12}} & \multicolumn{1}{l}{\textbf{0.42}} & \multicolumn{1}{l}{\textbf{177.69}} & \multicolumn{1}{l}{\textbf{5.30}} & \multicolumn{1}{l}{\textbf{1}} & \multicolumn{1}{l}{\textbf{9}} & \multicolumn{1}{l}{\textbf{2974.90}} & \multicolumn{1}{l}{\textbf{38.70}} \\
    \multicolumn{1}{l}{\textbf{50}} & \multicolumn{1}{l}{\textbf{25}} & \multicolumn{1}{l}{\textbf{3}} & \multicolumn{1}{l}{\textbf{105.58}} & \multicolumn{1}{l}{\textbf{0.14}} & \multicolumn{1}{l}{\textbf{893.42}} & \multicolumn{1}{l}{\textbf{8.89}} & \multicolumn{1}{l}{\textbf{1}} & \multicolumn{1}{l}{\textbf{20}} & \multicolumn{1}{l}{\textbf{2837.16}} & \multicolumn{1}{l}{\textbf{30.00}} \\
    \multicolumn{1}{l}{\textbf{50}} & \multicolumn{1}{l}{\textbf{25}} & \multicolumn{1}{l}{\textbf{4}} & \multicolumn{1}{l}{\textbf{139.79}} & \multicolumn{1}{l}{\textbf{0.13}} & \multicolumn{1}{l}{\textbf{705.29}} & \multicolumn{1}{l}{\textbf{11.16}} & \multicolumn{1}{l}{\textbf{1}} & \multicolumn{1}{l}{\textbf{25}} & \multicolumn{1}{l}{\textbf{2150.84}} & \multicolumn{1}{l}{\textbf{36.82}} \\
    \multicolumn{1}{l}{\textbf{50}} & \multicolumn{1}{l}{\textbf{25}} & \multicolumn{1}{l}{\textbf{5}} & \multicolumn{1}{l}{\textbf{75.51}} & \multicolumn{1}{l}{\textbf{0.10}} & \multicolumn{1}{l}{\textbf{396.18}} & \multicolumn{1}{l}{\textbf{15.47}} & \multicolumn{1}{l}{\textbf{1}} & \multicolumn{1}{l}{\textbf{34}} & \multicolumn{1}{l}{\textbf{1056.32}} & \multicolumn{1}{l}{\textbf{30.85}} \\
    \multicolumn{1}{l}{\textbf{50}} & \multicolumn{1}{l}{\textbf{25}} & \multicolumn{1}{l}{\textbf{6}} & \multicolumn{1}{l}{\textbf{32.18}} & \multicolumn{1}{l}{\textbf{0.23}} & \multicolumn{1}{l}{\textbf{196.78}} & \multicolumn{1}{l}{\textbf{11.75}} & \multicolumn{1}{l}{\textbf{1}} & \multicolumn{1}{l}{\textbf{30}} & \multicolumn{1}{l}{\textbf{896.85}} & \multicolumn{1}{l}{\textbf{34.41}} \\
    \multicolumn{1}{l}{\textbf{50}} & \multicolumn{1}{l}{\textbf{100}} & \multicolumn{1}{l}{\textbf{2}} & \multicolumn{1}{l}{\textbf{339.41}} & \multicolumn{1}{l}{\textbf{52.36}} & \multicolumn{1}{l}{\textbf{919.97}} & \multicolumn{1}{l}{\textbf{5.50}} & \multicolumn{1}{l}{\textbf{2}} & \multicolumn{1}{l}{\textbf{9}} & \multicolumn{1}{l}{\textbf{2871.50}} & \multicolumn{1}{l}{\textbf{44.66}} \\
    \multicolumn{1}{l}{\textbf{60}} & \multicolumn{1}{l}{\textbf{80}} & \multicolumn{1}{l}{\textbf{3}} & \multicolumn{1}{l}{\textbf{300.23}} & \multicolumn{1}{l}{\textbf{0.97}} & \multicolumn{1}{l}{\textbf{1379.76}} & \multicolumn{1}{l}{\textbf{7.20}} & \multicolumn{1}{l}{\textbf{1}} & \multicolumn{1}{l}{\textbf{16}} & \multicolumn{1}{l}{\textbf{1783.35}} & \multicolumn{1}{l}{\textbf{37.13}} \\
    \hline
    \end{tabular}%
  \label{tab:addlabel}%
\end{table}
\end{landscape}

\newpage
\begin{landscape}
\begin{table}[htbp]
  \centering
  \caption{Big size problems}
    \begin{tabular}{ccccccccccc}
    \hline
    \multicolumn{1}{c}{\textbf{n}} & \multicolumn{1}{c}{\textbf{m}} & \multicolumn{1}{c}{\textbf{r}} & \multicolumn{1}{c}{\textbf{Mean CPU(s)}} & \multicolumn{1}{c}{\textbf{Min CPU(s)}} & \multicolumn{1}{c}{\textbf{Max CPU(s)}} & \multicolumn{1}{c}{\textbf{Mean Eff Sol}} & \multicolumn{1}{c}{\textbf{Min Eff Sol}} & \multicolumn{1}{c}{\textbf{Max Eff Sol}} & \multicolumn{1}{c}{\textbf{Mean CN}} & \multicolumn{1}{c}{\textbf{Mean SN/CN \%}} \\
    \hline
    \multicolumn{1}{l}{\textbf{80}} & \multicolumn{1}{l}{\textbf{35}} & \multicolumn{1}{l}{\textbf{2}} & \multicolumn{1}{l}{\textbf{534.24}} & \multicolumn{1}{l}{\textbf{0.28}} & \multicolumn{1}{l}{\textbf{3117.52}} & \multicolumn{1}{l}{\textbf{4.37}} & \multicolumn{1}{l}{\textbf{1}} & \multicolumn{1}{l}{\textbf{13}} & \multicolumn{1}{l}{\textbf{9011.00}} & \multicolumn{1}{l}{\textbf{41.69}} \\
    \multicolumn{1}{l}{\textbf{80}} & \multicolumn{1}{l}{\textbf{35}} & \multicolumn{1}{l}{\textbf{3}} & \multicolumn{1}{l}{\textbf{672.09}} & \multicolumn{1}{l}{\textbf{0.35}} & \multicolumn{1}{l}{\textbf{4717.27}} & \multicolumn{1}{l}{\textbf{9.95}} & \multicolumn{1}{l}{\textbf{1}} & \multicolumn{1}{l}{\textbf{25}} & \multicolumn{1}{l}{\textbf{6907.16}} & \multicolumn{1}{l}{\textbf{38.17}} \\
    \multicolumn{1}{l}{\textbf{80}} & \multicolumn{1}{l}{\textbf{35}} & \multicolumn{1}{l}{\textbf{4}} & \multicolumn{1}{l}{\textbf{266.54}} & \multicolumn{1}{l}{\textbf{0.41}} & \multicolumn{1}{l}{\textbf{1013.70}} & \multicolumn{1}{l}{\textbf{13.28}} & \multicolumn{1}{l}{\textbf{1}} & \multicolumn{1}{l}{\textbf{41}} & \multicolumn{1}{l}{\textbf{2754.83}} & \multicolumn{1}{l}{\textbf{34.65}} \\
    \multicolumn{1}{l}{\textbf{80}} & \multicolumn{1}{l}{\textbf{35}} & \multicolumn{1}{l}{\textbf{6}} & \multicolumn{1}{l}{\textbf{198.58}} & \multicolumn{1}{l}{\textbf{0.25}} & \multicolumn{1}{l}{\textbf{1161.67}} & \multicolumn{1}{l}{\textbf{15.58}} & \multicolumn{1}{l}{\textbf{1}} & \multicolumn{1}{l}{\textbf{46}} & \multicolumn{1}{l}{\textbf{1295.68}} & \multicolumn{1}{l}{\textbf{34.45}} \\
    \multicolumn{1}{l}{\textbf{80}} & \multicolumn{1}{l}{\textbf{50}} & \multicolumn{1}{l}{\textbf{3}} & \multicolumn{1}{l}{\textbf{504.61}} & \multicolumn{1}{l}{\textbf{1.31}} & \multicolumn{1}{l}{\textbf{1836.49}} & \multicolumn{1}{l}{\textbf{8.85}} & \multicolumn{1}{l}{\textbf{1}} & \multicolumn{1}{l}{\textbf{23}} & \multicolumn{1}{l}{\textbf{5397.30}} & \multicolumn{1}{l}{\textbf{36.26}} \\
    \multicolumn{1}{l}{\textbf{80}} & \multicolumn{1}{l}{\textbf{120}} & \multicolumn{1}{l}{\textbf{2}} & \multicolumn{1}{l}{\textbf{1960.21}} & \multicolumn{1}{l}{\textbf{14.18}} & \multicolumn{1}{l}{\textbf{7509.54}} & \multicolumn{1}{l}{\textbf{6.00}} & \multicolumn{1}{l}{\textbf{1}} & \multicolumn{1}{l}{\textbf{13}} & \multicolumn{1}{l}{\textbf{4172.10}} & \multicolumn{1}{l}{\textbf{43.44}} \\
    \multicolumn{1}{l}{\textbf{100}} & \multicolumn{1}{l}{\textbf{80}} & \multicolumn{1}{l}{\textbf{3}} & \multicolumn{1}{l}{\textbf{1171.52}} & \multicolumn{1}{l}{\textbf{2.94}} & \multicolumn{1}{l}{\textbf{4490.26}} & \multicolumn{1}{l}{\textbf{8.05}} & \multicolumn{1}{l}{\textbf{3}} & \multicolumn{1}{l}{\textbf{18}} & \multicolumn{1}{l}{\textbf{5200.53}} & \multicolumn{1}{l}{\textbf{36.41}} \\
    \multicolumn{1}{l}{\textbf{120}} & \multicolumn{1}{l}{\textbf{60}} & \multicolumn{1}{l}{\textbf{4}} & \multicolumn{1}{l}{\textbf{1338.38}} & \multicolumn{1}{l}{\textbf{3.32}} & \multicolumn{1}{l}{\textbf{6259.84}} & \multicolumn{1}{l}{\textbf{11.65}} & \multicolumn{1}{l}{\textbf{3}} & \multicolumn{1}{l}{\textbf{28}} & \multicolumn{1}{l}{\textbf{3562.95}} & \multicolumn{1}{l}{\textbf{35.75}} \\
    \multicolumn{1}{l}{\textbf{200}} & \multicolumn{1}{l}{\textbf{100}} & \multicolumn{1}{l}{\textbf{3}} & \multicolumn{1}{l}{\textbf{}} & \multicolumn{1}{l}{\textbf{35801.05}} & \multicolumn{1}{l}{\textbf{}} & \multicolumn{1}{l}{\textbf{}} & \multicolumn{1}{l}{\textbf{14}} & \multicolumn{1}{l}{\textbf{}} & \multicolumn{1}{l}{\textbf{}} & \multicolumn{1}{l}{\textbf{}} \\
    \hline
    \end{tabular}%
  \label{tab:addlabel}%
\end{table}
\end{landscape}

The results clearly show that the number of generated efficient solutions grows with the number of criteria. However, the proposed method does not review all the efficient solution set of the MOIFP problem to find the optimal solution of the problem (P), only a few of them are visited. Indeed, in \cite{9} the authors reported that, for example for instances with 25 variables, 10 constraints and 4 criteria, up to 100 efficient solutions on average are found for the MOILFP problem. 
The increase in the criteria number does not seem to have a negative effect, as well on the CPU time as the generated nodes number in the search tree, these can even decrease. 
Note that the minimum number of generated efficient solutions is generally equal to 1. In fact, this corresponds to two types of instances:
\begin{itemize}
\item Those for which the first solution found is integer, efficient for MOILFP and optimal for the problem to be solved, resolved in less than one second;
\item Those for which the program has generated integer feasible solutions, but only one is efficient for MOILFP, and therefore optimal for the problem to be solved.
\end{itemize}

We notice that on average, the number of sterilized nodes is around 35\% compared to those created in the search tree. This is explained by the addition of cuts of type I and II, combined with the evaluation test using the local ideal point.

\section{Conclusion}
In this paper, a new exact method is presented to solve the problem of optimizing a linear  fractional  function  over  the  set of integer  efficient solutions of  a  MOILFP problem without generating all of them. It is a branch and bound method coupled with cutting planes having the effect to delete some regions of the feasible domain not containing efficient solutions for the MOILFP problem nor optimal solutions for the fractional function to optimize over the efficient set. In addition to the lower bound of the function to be optimized, two vector bounds are introduced to prune considerably nodes in the tree search; the local ideal point of the sub domain of a node as an upper bound and the set of potentially efficient solutions generated along the resolution process as a set of lower bounds.  The experimental results show that the number of criteria of the MOILFP problem is not a handicap for the method which can take several linear fractional criteria simultaneously. However, it is not the case for the number of variables which weighs heavy on the CPU time.
Being able to extend the method to optimize other types of functions on an integer efficient set seems like a challenge.

%
%
%
 \bibliographystyle{splncs04}
%

\end{document}